\documentclass{article}
\usepackage{amsmath,amssymb,theorem}
\usepackage{graphicx}
\usepackage{verbatim}

\usepackage{color}
\usepackage{makeidx}

\theorembodyfont{\upshape}

 \makeatletter
 \@addtoreset{equation}{section}
 \makeatother

 \newtheorem{ittheorem}{Theorem}
 \newtheorem{itlemma}{Lemma}
 \newtheorem{itproposition}{Proposition}
 \newtheorem{itdefinition}{Definition}

 \newtheorem{itremark}{Remark}
 \newtheorem{itclaim}{Claim}
 \newtheorem{itcorollary}{\bf Corollary}

 \newenvironment{theorem}{\addtocounter{equation}{1}
 \begin{ittheorem}}{\end{ittheorem}}

 \newenvironment{lemma}{\addtocounter{equation}{1}
 \begin{itlemma}}{\end{itlemma}}

 \newenvironment{proposition}{\addtocounter{equation}{1}
 \begin{itproposition}}{\end{itproposition}}

 \newenvironment{definition}{\addtocounter{equation}{1}
 \begin{itdefinition}}{\end{itdefinition}}

 \newenvironment{remark}{\addtocounter{equation}{1}
 \begin{itremark}}{\end{itremark}}

 \newenvironment{claim}{\addtocounter{equation}{1}
 \begin{itclaim}}{\end{itclaim}}

 \newenvironment{proof}{\noindent {\bf Proof.\,}
 }{\hspace*{\fill}$\qed$\medskip}

 \newenvironment{corollary}{\addtocounter{equation}{1}
 \begin{itcorollary}}{\end{itcorollary}}

 \newcommand{\be}[1]{\begin{eqnarray*}\label{#1}}
 \newcommand{\ee}{\end{eqnarray*}}

 \newcommand{\bl}[1]{\begin{lemma}\label{#1}}
 \newcommand{\el}{\end{lemma}}

 \newcommand{\br}[1]{\begin{remark}\label{#1}}
 \newcommand{\er}{\end{remark}}

 \newcommand{\bt}[1]{\begin{theorem}\label{#1}}
 \newcommand{\et}{\end{theorem}}

 \newcommand{\bd}[1]{\begin{definition}\label{#1}}
 \newcommand{\ed}{\end{definition}}

 \newcommand{\bcl}[1]{\begin{claim}\label{#1}}
 \newcommand{\ecl}{\end{claim}}

 \newcommand{\bp}[1]{\begin{proposition}\label{#1}}
 \newcommand{\ep}{\end{proposition}}

 \newcommand{\bc}[1]{\begin{corollary}\label{#1}}
 \newcommand{\ec}{\end{corollary}}

 \newcommand{\bpr}{\begin{proof}}
 \newcommand{\epr}{\end{proof}}

 \newcommand{\bi}{\begin{itemize}}
 \newcommand{\ei}{\end{itemize}}

 \newcommand{\ben}{\begin{enumerate}}
 \newcommand{\een}{\end{enumerate}}

\def\uro{\smash{{U}^{\!\!\!\!\raise5pt\hbox{$\scriptstyle o$}}}}

\def\bp{{\overline{p}}}

\def\bp{{\overline{p}}}

 \def \ba {\begin{array}}
 \def \ea {\end{array}}

 \def \qed {{\heartsuit\hfill}}

 \def \N {{\mathbb N}}

\def \qed {{\square\hfill}}

\def \qed {{\square\hfill}}

\def\N{{\mathbb N}}

\def\eqref#1{(\ref{#1})}

\def\card{\text{card}\,}

\makeindex

\begin{document}
\allowdisplaybreaks[4]

\title{A probabilistic proof of Perron's theorem}

\author{
\qquad Rapha\"el Cerf \hskip 70pt Joseba Dalmau\\
{\small DMA, {\'E}cole Normale Sup\'erieure \qquad CMAP, Ecole Polytechnique}
}

\maketitle



\begin{abstract}
\noindent
We present an alternative proof of Perron's theorem, which is probabilistic in nature.
It rests 
on the representation of the Perron eigenvector as
a functional of the trajectory of an auxiliary Markov chain. This formula generalises the well--known
formula for the invariant probability measure of a finite state space Markov chain.
\end{abstract}



\bigskip

\noindent
In 1907, Oskar Perron proved the following theorem.

\noindent
\begin{theorem}
Let $A$ be a square matrix with positive entries.
Then 
the matrix $A$ admits a positive eigenvalue $\lambda$ such that:


\noindent
\phantom{i}i) 
to $\lambda$ is associated an 
eigenvector $\mu$
whose components are all positive;

\noindent
\phantom{i}ii) if $\alpha$ is another eigenvalue of $A$, possibly complex, then $|\alpha|< \lambda$;

\noindent
\phantom{}iii) any other eigenvector associated to $\lambda$ is a multiple of $\mu$.

\end{theorem}
This theorem was subsequently generalised by Frobenius in his work on non--negative matrices
in 1912, leading to the so--called Perron--Frobenius theorem \cite{SE}.
A myriad of mathematical models involve non--negative matrices and their powers,
thereby calling for the use of the Perron--Frobenius theorem.
Mathematicians have developed generalisations in several directions, notably
in infinite dimensions (for infinite matrices \cite{VJ}, for non--negative kernels
in arbitrary spaces \cite{AN}) and a whole Perron--Frobenius theory has emerged.
Hawkins wrote an historical account on the initial development of this theory \cite{HA}.
MacCluer \cite{MC} describes several applications of Perron's theorem and reviews the different proofs
that have been found over the years.
The original proof of Perron rested on an induction over the size of the matrix.
A few years later Perron found a proof involving the resolvent of the matrix.
A nowadays popular proof, which is found in most textbooks, is due to Wielandt
and it rests on a miraculous max--min functional.

We present an alternative proof of Perron's theorem, which is probabilistic in nature.
It rests on an auxiliary Markov chain, and the representation of the Perron eigenvector as
a functional of the trajectory of this Markov chain. This formula generalises the well--known
formula for the invariant probability measure of a finite state space Markov chain.
To ease the exposition, we restrict ourselves to the Perron theorem, and we work with matrices
whose entries are all positive. However our proof can be readily extended to primitive matrices,
thereby yielding the classical
Perron--Frobenius theorem.
Our proof might seem lengthy compared to other proofs, yet it is completely self--contained
and it requires only classical results of basic algebra and power series.

We introduce next some notation in order to define the auxiliary Markov chain. 
Let $d$ be a positive integer. 
Throughout the text, we consider a square 
matrix $A=(A(i,j))_{1\leq i,j\leq d}$ of size $d\times d$ with positive entries.
For $i\in\{\,1,\dots,d\,\}$, we denote by $S(i)$ the sum of the entries on the $i$--th row
of $A$, i.e.,
$$\forall i\in
\{\,1,\dots,d\,\}\qquad S(i)\,=\,\sum_{j=1}^{d}A(i,j)\,,$$
and we create a new
matrix 
$M=(M(i,j))_{1\leq i,j\leq d}$ 
by setting
$$\forall i,j\in
\{\,1,\dots,d\,\}\qquad 
M(i,j)\,=\,\frac{A(i,j)}{S(i)}\,.$$
Obviously, the sum of each row of $M$ is now equal to one, i.e., $M$ is stochastic,
and we think of it as the transition matrix of a Markov chain.
So, 
let $(X_n)_{n\in\N}$ be a Markov chain with state space
$\lbrace\,1,\dots,d\,\rbrace$ and transition matrix $M$.
Let us fix  $i\in\{\,1,\dots,d\,\}$.
We denote by $E_i$ the expectation of the Markov chain issued from $i$
and we introduce
the time 
$\tau_i$ 
of the first return of the chain to $i$, defined by
$$\tau_i\,=\,\inf\,\big\{\,n\geq 1:X_n=i\,\big\}\,.$$
Finally, we define a 
function $\phi_i$ by setting
$$\forall \lambda\geq 0\qquad
\phi_i(\lambda)\,=\,
E_i\Bigg(
\lambda^{-\tau_i}\prod_{n=0}^{\tau_i-1}S(X_n)
\Bigg)
\,.$$
The quantity in the expectation is non--negative, so the function $\phi_i$
is well defined and it might take infinite values.
In fact, the function $\phi_i$ can be written as a power series in the variable $1/\lambda$,
as follows:
\begin{multline*}
\phi_i(\lambda)\,=\,
\sum_{k=1}^{\infty}\frac{1}{\lambda^k}
E_i\Bigg(
1_{\{\tau_i=k\}}
\prod_{n=0}^{k-1}S(X_n)\Bigg)\,=\,\cr
\sum_{k=1}^{\infty}\frac{1}{\lambda^k}
\kern-3pt
\sum_{i_1,\dots,i_{k-1}\neq i}
\kern-7pt
S(i)S(i_1)\cdots S(i_{k-1})
\hfill\cr \hfill\times
P\big(X_1=i_1,\dots,X_{k-1}=i_{k-1},X_k=i\,|\,X_0=i\big)\cr
\,=\,
\sum_{k=1}^{\infty}\frac{1}{\lambda^k}
\sum_{i_1,\dots,i_{k-1}\neq i}
S(i)M(i,i_1)\cdots S(i_{k-1})M(i_{k-1},i)\cr
\,=\,
\sum_{k=1}^{\infty}\frac{1}{\lambda^k}
\sum_{i_1,\dots,i_{k-1}\neq i}
A(i,i_1)\cdots A(i_{k-1},i)\,.
\end{multline*}
Let $R$ be the radius of the convergence circle of this series, seen as a power series in the variable $1/\lambda$.
\begin{proposition}
\label{regf}
The function $\phi_i$ is continuous,
	decreasing on $]1/R,+\infty[$ and it satisfies
$$\lim_{\lambda\to 1/R\atop\lambda>1/R}\phi_i(\lambda)\,=\,+\infty\,,\qquad
	\phi_i(1/R)\,=\,+\infty\,,\qquad
\lim_{\lambda\to +\infty}\phi_i(\lambda)\,=\,0\,.$$
\end{proposition}
\begin{proof}
Since $A$ has positive entries,
the series contains non vanishing terms, and this implies that
$\phi_i$ is decreasing and tends to $\infty$ as $\lambda$ goes to~$1/R$ from above.
From classical results on power series, we know that $\phi_i(\lambda)$ is
continuous for $\lambda>1/R$.
Let us show that
$\phi_i(1/R)=+\infty$.
Let $B$ be the matrix obtained from $A$ by removing the $i$--th row and the $i$--th column and let
$\gamma_1,\dots,\gamma_{d-1}$ be its eigenvalues (possibly complex), arranged so that
$|\gamma_1|\geq\cdots
\geq |\gamma_d|$.
Let $m$ (respectively $M$) be the minimum (respectively the maximum) of the entries of $A$.
For any $k\geq 1$, we have
\begin{multline*}
\sum_{i_1,\dots,i_{k-1}\neq i}
A(i,i_1)\cdots A(i_{k-1},i)\,\geq\,
\frac{m^2}{M}
\sum_{i_1,\dots,i_{k-1}\neq i}
A(i_1,i_2)\cdots A(i_{k-1},i_1)
\cr
\,=\,
	\frac{m^2}{M} \text{trace}(B^{k-1})
\,=\,
	\frac{m^2}{M} \Big(\gamma_1^{k-1}+\cdots+\gamma_{d-1}^{k-1}\Big)
\,.
\end{multline*}
Although the eigenvalues
	$\gamma_1,\dots,\gamma_{d-1}$ might be complex numbers, the trace of $B^{k-1}$
is a positive real number.
Similarly, we have
$$
\sum_{i_1,\dots,i_{k-1}\neq i}
A(i,i_1)\cdots A(i_{k-1},i)\,\leq\,
	\frac{M^2}{m} \Big(\gamma_1^{k-1}+\cdots+\gamma_{d-1}^{k-1}\Big)
\,.
$$
From the two previous inequalities, we conclude that the power series defining $\phi_i$
converges
if and only if the series
$$\sum_{k=0}^{\infty}\frac{1}{\lambda^k}
\Big(\gamma_1^k+\cdots+\gamma_{d-1}^k\Big)
$$
converges.
This is certainly the case if $|\lambda|>|\gamma_1|$, therefore $R\geq 1/|\gamma_1|$.
Let us define, for $n\geq 1$,
$$S_n(\lambda)\,=\,
\sum_{k=0}^{n}\frac{1}{\lambda^k}
\Big(\gamma_1^k+\cdots+\gamma_{d-1}^k\Big)\,.
$$
We shall rely on the following result on geometric series.
\begin{lemma}
\label{geo}
Let $z$ be a complex number such that $|z|\leq 1$. Then
$$
\lim_{n\to\infty}\,\,
\frac{1}{n}\big({1+z+\cdots+z^n}\big)\,=\,
\begin{cases}
\quad0&\quad\text{if}\quad z\neq 1\,,\\
\quad1&\quad\text{if}\quad z=1\,.
\end{cases}
$$
\end{lemma}
\begin{proof} For $z=1$, the result is obvious. For $z\neq 1$, we compute
$$\frac{1}{n}\big({1+z+\cdots+z^n}\big)\,=\,
\frac{1-z^{n+1}}{n(1-z)}\,,$$
and we observe that this quantity goes to $0$ when $n$ goes to $\infty$.
\end{proof}

\noindent
Lemma~\ref{geo} implies that, for $\lambda$ a complex number such that $|\lambda|=|\gamma_1|$,
$$\lim_{n\to\infty}\,\,
\frac{1}{n}S_n(\lambda)\,=\,
\card\big\{\,j:1\leq j\leq d,\,\lambda=\gamma_j\,\big\}\,.$$
This implies in particular that
$$\lim_{n\to\infty}\,\,
\big|S_n(\gamma_1)\big|\,=\,+\infty\,.$$
Observing that
$\big|S_n(\gamma_1)\big|
\leq
S_n(|\gamma_1|)$,
we conclude that 
$$\phi_i(|\gamma_1|)\,=\,\lim_{n\to\infty}\,\,
S_n(|\gamma_1|)\,=\,+\infty\,.$$
Therefore $R=1/|\gamma_1|$ and moreover
$\phi_i(1/R)=+\infty$.
\end{proof}

\noindent
Proposition~\ref{regf} implies that
$\phi_i$ is one to one from $]1/R,+\infty[$ onto $]0,+\infty[$, thus
there exists a unique positive 
real number $\lambda_i$ such that $\phi_i(\lambda_i)=1$.
The next result is the key to our proof of
the Perron--Frobenius theorem.
We define a vector $\mu_i$ by setting
$$\forall j\in \{\,1,\dots,d\,\}\qquad
\mu_i(j)\,=\,
\displaystyle
E_i\Bigg(\sum_{n=0}^{\tau_i-1}\Big(
1_{\{X_n=j\}}
\lambda_i^{-n}\prod_{k=0}^{n-1}S(X_k)
\Big)
\Bigg)
\,.
$$
\begin{theorem}
The value $\lambda_i$ is an eigenvalue of $A$ and
the vector $\mu_i$ is an associated left eigenvector whose components are all positive and finite.
\end{theorem}
\begin{proof}
Let us note $E_i,\tau_i,\lambda_i,\mu_i$ simply by $E,\tau,\lambda,\mu$.
	Let us compute, for $k\in\{\,1,\dots,d\,\}$,
\begin{align*}
\sum_{j=1}^d
\mu(j) &A(j,k)\,=\,
\sum_{j=1}^d
\mu(j) S(j)M(j,k)\cr
&\,{=}\,\sum_{j=1}^d\sum_{n\geq 0}
E\Bigg(1_{\{\tau>n\}}
\lambda^{-n}
\Big(
\prod_{t=0}^{n-1}S(X_t)
\Big)
1_{\{X_n=j\}}
f(j)M(j,k)
\Bigg)\cr
&=\,\sum_{j=1}^d\sum_{n\geq 0}
E\Bigg(1_{\{\tau>n\}}
\lambda^{-n}
\Big(
\prod_{t=0}^{n}S(X_t) \Big)
1_{\{X_n=j\}}
1_{\{X_{n+1}=k\}}
\Bigg)\cr
&=\,
E\Bigg(\sum_{n=0}^{\tau-1}
1_{\{X_{n+1}=k\}}
\lambda^{-n}
\Big(
\prod_{t=0}^{n}S(X_t)
\Big)
\Bigg)\cr
&=\,
\lambda\,E\Bigg(\sum_{n=1}^{\tau}
1_{\{X_{n}=k\}}
\lambda^{-n}
\Big(
\prod_{t=0}^{n-1}S(X_t)
\Big)
\Bigg)\,.
\end{align*}
Suppose that $k\neq i$. Then the term in the last sum vanishes 
for $n=0$ or $n=\tau$, and we obtain
$$
\sum_{j=1}^d 
\mu(j) A(j,k)\,=\,
\lambda \mu(k)\,.$$
For $k=i$, the only non--vanishing term in the expectation corresponds to $n=\tau$ and we obtain, noticing that $\mu(i)=1$,
$$
\sum_{j=1}^d 
\mu(j) A(j,i)\,=\,
\lambda\,E\Bigg(
\lambda^{-\tau}
\prod_{t=0}^{\tau-1}S(X_t)
\Bigg)
\,=\,
\lambda\,\phi_i(\lambda)\,=\,
\lambda\,\mu(i)\,.
$$
Thus we have proved that
$\mu A=\lambda\mu$.
Since $\mu(i)=1$, 
these equations
imply that 
$\mu(1),\dots,\mu(d)$ are all positive and finite. 
\end{proof}
\begin{proposition}
\label{regu}
Let $\alpha$ be an eigenvalue of $A$, possibly complex, and let
$\nu$ be an associated left eigenvector. 
Let $i \in\{\,1,\dots,d\,\}$ be such that $\nu(i)\neq 0$.
Either 
$\nu$ and $\mu_i$ are proportional (in which case 
$\alpha=\lambda_i$) or
$|\alpha|<\lambda_i$.
\end{proposition}
\begin{proof}
Let $\alpha,\nu$ and $i$ be as in the statement of the proposition.
We suppose that $\alpha\neq 0$,
otherwise there is nothing to prove. Let $\nu$ be an associated left eigenvector. 
We have 
$$
\forall k\in\{\,1,\dots,d\,\}\qquad
\nu(k)
\,=\,
\frac{1}{\alpha}
\sum_{j=1}^d 
\nu(j) A(j,k)
\,.
$$
Let us focus on
the equation for $k=i$. We divide by $\nu(i)$ (which is assumed to be non zero) 
and
we isolate the term $j=i$ in the sum to obtain
$$
1\,=\,
\frac{1}{\alpha}
 A(i,i)+
\frac{1}{\alpha}
\sum_{j\neq i} 
\frac{\nu(j)}{\nu(i)} A(j,i)
\,.
$$
We expand $\nu(j)$ in the above equation as a sum, and we get
\begin{multline*}
1
\,=\,
\frac{1}{\alpha}
 A(i,i)+
\frac{1}{\alpha^2}
\sum_{j\neq i} 
\sum_{j'} 
\frac{\nu(j')}{\nu(i)}
 A(j',j)A(j,i)
\cr
\,=\,
\frac{1}{\alpha}
 A(i,i)+
\frac{1}{\alpha^2}
\sum_{j\neq i} 
 A(i,j)A(j,i)
+
\frac{1}{\alpha^2}
\sum_{j\neq i} 
\sum_{j'\neq i} 
\frac{\nu(j')}{\nu(i)}
 A(j',j)A(j,i)
\,.
\end{multline*}
Iterating $n$ times this procedure, we get
\begin{multline*}
1
\,=\,
\frac{1}{\alpha}
 A(i,i)+\cdots+
\frac{1}{\alpha^{n+1}}
\sum_{i_1,\dots,i_{n}\neq i} 
 A(i,i_1)
 A(i_1,i_2)\cdots
 A(i_{n},i)
\cr
+
\frac{1}{\alpha^{n+1}}
\sum_{i_0,i_1,\dots,i_{n}\neq i} 
\frac{\nu(i_0)}{\nu(i)}
 A(i_0,i_1)
 A(i_1,i_2)\cdots
 A(i_{n},i)
\,.
\end{multline*}
If $\phi_i(|\alpha|)=+\infty$, then it follows from proposition~\ref{regf} and
the definition of $\lambda_i$ that
$|\alpha|\leq 1/R<\lambda_i$ and we are done.
From now onwards, we suppose that 
$\phi_i(|\alpha|)<+\infty$.
Just before 
proposition~\ref{regf}, 
we worked out a power series expansion of $\phi_i$. The convergence of this series
at $|\alpha|$ implies in particular that the general
term of this series goes to $0$, hence
$$
\lim_{n\to\infty}\,\,
\frac{1}{\alpha^{n+1}}
\sum_{i_1,\dots,i_{n}\neq i} 
 A(i,i_1)
 A(i_1,i_2)\cdots
 A(i_{n},i)
\,=\,0
\,.
$$
Let $m$ (respectively $M$) be the minimum (respectively the maximum) of the entries of $A$.
For any $i_0\neq i$, we have
\begin{multline*}
\sum_{i_1,\dots,i_{n}\neq i} 
 A(i_0,i_1)
 A(i_1,i_2)\cdots
 A(i_{n},i)
\,\leq\,
\cr
\frac{M}{m}
\sum_{i_1,\dots,i_{n}\neq i} 
 A(i,i_1)
 A(i_1,i_2)\cdots
 A(i_{n},i)\,.
\end{multline*}
It follows that, for any $n\geq 1$,
\begin{multline*}
\frac{1}{|\alpha|^{n+1}}
\sum_{i_0,i_1,\dots,i_{n}\neq i} 
\frac{\nu(i_0)}{\nu(i)}
 A(i_0,i_1)
 A(i_1,i_2)\cdots
 A(i_{n},i)
\,\leq\,\cr
\frac{Md
\max_{1\leq j\leq d}|\nu(j)|
}{m|\nu(i)|}
\frac{1}{|\alpha|^{n+1}}
\sum_{i_1,\dots,i_{n}\neq i} 
 A(i,i_1)
 A(i_1,i_2)\cdots
 A(i_{n},i)
\end{multline*}
and we conclude from the previous inequality that
$$
\lim_{n\to\infty}\,\,
\frac{1}{\alpha^{n+1}}
\sum_{i_0,i_1,\dots,i_{n}\neq i} 
\frac{\nu(i_0)}{\nu(i)}
 A(i_0,i_1)
 A(i_1,i_2)\cdots
 A(i_{n},i)
\,=\,0\,.$$
We send now $n$ to $\infty$ in the identity and we get
$$
1
\,=\,
\frac{1}{\alpha}
 A(i,i)+\sum_{n=1}^{+\infty}
\frac{1}{\alpha^{n+1}}
\sum_{i_1,\dots,i_{n}\neq i} 
 A(i,i_1)
 A(i_1,i_2)\cdots
 A(i_{n},i)
\,.$$
Recall that $\alpha$ might be complex. Taking the modulus, we conclude
that $\phi_i(|\alpha|)\geq 1=\phi_i(\lambda_i)$, and since $\phi_i$ is decreasing, then $|\alpha|\leq\lambda_i$.
It remains to examine the case 
$|\alpha|=\lambda_i$.
We suppose that the eigenvector $\nu$ associated to $\alpha$
is normalised so that $\nu(i)=1$.
We denote by $|\nu|$ the vector whose coordinates are the modulus of the coordinates of
$\nu$, i.e., 
$|\nu|(j)=
|\nu(j)|$ for $1\leq j\leq d$.
Since $\nu A=\alpha\nu$ and the entries of $A$ are positive, then
$$
\forall k\in\{\,1,\dots,d\,\}\qquad
|\nu|(k)
\,\leq\,
\frac{1}{\lambda_i}
\sum_{j=1}^d 
|\nu|(j) A(j,k)
\,.
$$
Starting from this inequality, 
we proceed as previously,
that is, we isolate the term corresponding
to $j=i$ in the sum, we bound from above the term $|\nu(j)|$ for $j\neq i$ with the help of
the above inequality, 
and we iterate the procedure $n$ times. We check that the ultimate term goes to $0$
when we send $n$ to $\infty$, and we get the inequality 
$$
\forall k\in\{\,1,\dots,d\,\}\qquad
|\nu|(k)
\,\leq\,
\mu_i(k)\,.$$
For $k\in\{\,1,\dots,d\,\}$,
we have
$$
{\lambda_i}
|\nu|(k)\,=\,
\Big|
\sum_{j=1}^d 
\nu(j) A(j,k)
\Big|
\,\leq\,
\sum_{j=1}^d 
|\nu|(j) A(j,k)
\,.
$$
It follows that
$$
\sum_{k=1}^d 
\big(\mu_i(k)
-
|\nu|(k)\big)
A(k,i)\,\leq\,
\lambda_i\,\big(\mu(i)-\nu(i)\big)\,=\,0\,.
$$
This equation implies that
$\mu_i
=
|\nu|$ and that all the intermediate inequalities were in fact 
equalities. Since all the entries of $A$ are positive and $\nu(i)=1$, then
necessarily all the components of $\nu$ are non--negative real numbers and
$\nu=\mu_i$ and $\alpha=\lambda_i$.
\end{proof}

\noindent
The $\lambda_i$'s are positive eigenvalues of $A$, the eigenvectors $\mu_i$ have positive
coordinates, thus proposition~\ref{regu} readily
implies the following result.
\begin{corollary}
The values $\lambda_1,\dots,\lambda_d$ are all equal. 
Their common value $\lambda$ is a simple eigenvalue of $A$.
The eigenvectors $\mu_1,\dots,\mu_d$ are proportional.
\end{corollary}
Finally,
we normalise these eigenvectors by imposing that the sum of the components
is equal to~$1$, thereby getting a probability distribution.
\begin{corollary}
The left Perron--Frobenius eigenvector $\mu$ of $A$
is given by 
$$
\forall i\in\{\,1,\dots,d\,\}\qquad
\mu(i)\,=\,\frac{1}{
\displaystyle
E_i\Bigg(\sum_{n=0}^{\tau_i-1}\Big(
\lambda^{-n}\prod_{t=0}^{n-1}S(X_t)
\Big)
\Bigg)
}\,.
$$
\end{corollary}
This formula is a generalisation of the classical formula for
the invariant probability measure of a Markov chain.
Indeed, in the particular case where $A$ is stochastic,
$S$ is constant equal to 1, $\lambda$ is also equal to 1,
and the formula of the corollary becomes 
$$
\forall i\in\{\,1,\dots,d\,\}\qquad
\mu(i)\,=\,\frac{1}{
\displaystyle
E_i(\tau_i)
}\,,$$
a formula 
well--known among probabilists.
\noindent

\bibliographystyle{plain}
\bibliography{pamm}
 \thispagestyle{empty}

\end{document}